\documentclass[12pt]{article}
\usepackage[T1]{fontenc}
\usepackage{a4wide}
\usepackage{amsthm}
\usepackage{amsfonts}
\usepackage{amssymb}
\usepackage{amsmath}
\usepackage{stmaryrd}
\usepackage{cite}
\usepackage{epsfig}
\usepackage{tabularx}
\usepackage{mathtools}
\usepackage[size=scriptsize]{todonotes}
\usepackage{nameref,cleveref}
\usepackage{float}

\newtheorem{theorem}{Theorem}

\newtheorem{lemma}[theorem]{Lemma}

\newcommand\TT{{\cal T}}
\newcommand\II{{\cal I}}

\newcommand{\udt}[2]{#1\left\llbracket #2\right\rrbracket}

\DeclareMathOperator{\rank}{rank}
\DeclareMathOperator{\cd}{cd}
\DeclareMathOperator{\dd}{dd}
\DeclareMathOperator{\cl}{cl}
\DeclareMathOperator{\td}{td}
\DeclareMathOperator{\csd}{c^{*}\hspace{-3pt}d}

\DeclareTextCompositeCommand{\v}{OT1}{l}{l\nobreak\hspace{-.1em}'}

\DeclarePairedDelimiter\set{\{}{\}}
\DeclarePairedDelimiter\card{\lvert}{\rvert}
\DeclarePairedDelimiter\paren{(}{)}

\begin{document}

\title{Closure property of contraction-depth of matroids\thanks{The second and third authors were supported by the MUNI Award in Science and Humanities (MUNI/I/1677/2018) of the Grant Agency of Masaryk University. The second author was also supported by the project GA24-11098S of the Czech Science Foundation.}}

\author{Marcin Bria\'nski\thanks{Theoretical Computer Science Department, Faculty of Mathematics and Computer Science, Jagiellonian University, Krak\'ow, Poland. E-mail: \texttt{marcin.brianski@doctoral.uj.edu.pl}.}\and
        Daniel Kr{\'a}l'\thanks{Institute of Mathematics, Leipzig University, Augustusplatz 10, 04109 Leipzig, and Max Planck Institute for Mathematics in the Sciences, Inselstra{\ss}e 22, 04103 Leipzig, Germany. E-mail: {\tt daniel.kral@uni-leipzig.de}. Previous affiliation: Faculty of Informatics, Masaryk University, Botanick\'a 68A, 602 00 Brno, Czech Republic.}\and
        Ander Lamaison\thanks{Extremal Combinatorics and Probability Group (ECOPRO), Institute for Basic Science (IBS), Daejeon, South Korea. This author was also supported by IBS-R029-C4. E-mail: {\tt ander@ibs.re.kr}. Previous affiliation: Faculty of Informatics, Masaryk University, Botanick\'a 68A, 602 00 Brno, Czech Republic.}}
\date{} 
\maketitle

\begin{abstract}
Contraction$^*$-depth is a matroid depth parameter analogous to tree-depth of graphs.
We establish the matroid analogue of the classical graph theory result asserting that
the tree-depth of a graph $G$ is the minimum height of a rooted forest whose closure contains $G$
by proving the following for every matroid $M$ (except the trivial case when $M$ consists of loops and coloops only):
the contraction$^*$-depth of $M$ plus one
is equal to the minimum contraction-depth of a matroid containing $M$ as a restriction.
\end{abstract}

\section{Introduction}
\label{sec:intro}

Tree-depth is one of fundamental graph width parameters and
appears in various contexts in combinatorics, e.g.~\cite{BriJMM+23,HatJMP+24,CzeNP21,NorSW23,DvoW22},
and algorithmic applications, e.g.~\cite{NedPSW23,CheCD+21,PACE20,IwaOO18,PilW18,KusR23}.
The parameter can be defined in several equivalent ways,
which makes it particularly robust, and
its importance has recently grown even further
because of its close connection to foundational results on sparsity~\cite[Chapters 6 and 7]{NesO12}.
We are interested in the matroid analogue of this parameter---the contraction$^*$-depth.
This matroid parameter was introduced in~\cite{KarKLM17} and
its importance rose because of its surprising connection
to preconditioners in combinatorial optimizations~\cite{BriKKPS22,BriKKPS24,ChaCKKP20,ChaCKKP22}.
One such result asserts that
the smallest dual tree-depth of a matrix row-equivalent to a given constraint matrix in an integer programming instance
is equal to the the contraction$^*$-depth of the vector matroid formed by the columns of the matrix~\cite{ChaCKKP20,ChaCKKP22}, and
a matrix with the optimal dual tree-depth can be constructed algorithmically in fixed-parameter time (when parameterized
by the contraction$^*$-depth and the entry complexity of the input matrix).

Tree-depth is a very versatile parameter as it has many interchangeable definitions.
The most common way is to introduce it via closures of spanning forests, also
known as \emph{elimination forests}. Seen through this lens, the tree-depth of a graph $G$, denoted by $\td(G)$,
is the minimum height of a rooted forest whose closure contains the graph $G$ as a subgraph;
see \Cref{subsec:depth} for formal definitions of these notions.
Alternatively, one can define tree-depth in a recursive fashion.
\begin{itemize}
\item If the graph $G$ has at most one vertex, then $\td(G) = \card{V(G)}$.
\item If $G$ is not connected, then $\td(G)$ is the maximum tree-depth of a component of $G$.
\item Otherwise $\td(G)=1+\min\limits_{v\in V(G)}\td(G\setminus v)$,
      i.e., $\td(G)$ is one plus the minimum tree-depth of $G\setminus v$
      where the minimum is taken over all vertices $v$ of $G$.
\end{itemize}
A third classical way to define tree-depth is by using the notion of \emph{centered colorings}.
A centered coloring of a graph $G$ is any coloring such that every
connected subgraph $H$ of $G$ contains a vertex of a unique color.
The tree-depth of $G$ is the minimum number of colors needed in any centered coloring of $G$.

In the case of the most well-known graph width parameter---tree-width---there
is a well-developed theory of the matroid analogue---branch-width.
The theory comes with many structural and
algorithmic results, see e.g.~\cite{GavKO12,Hli03a,Hli03b,Hli06,HliO07,HliO08,JeoKO18},
which parallels many of the developments in graph theory related to tree-width, including the famous Courcelle's Theorem~\cite{Cou90}.
An obvious barrier to defining the matroid notion corresponding to graph tree-depth,
is the absence of a matroid concept completely analogous to that of a vertex,
particularly in the general case of non-representable matroids.
This was overcome in~\cite{KarKLM17} by a definition presented in Subsection~\ref{subsec:depth},
which is applicable to general matroids and
in the specific case of represented matroids coincides with the intuitive
approach of lifting the inductive definition of graph tree-depth
where factoring by one-dimensional subspaces is substituted for vertex deletions.
Throughout this paper, we refer to this matroid depth parameter as contraction$^*$-depth
in line with the terminology used in~\cite{BriKKPS22,BriKKPS24}
to avoid confusion with a different notion of branch-depth defined in~\cite{DevKO20}
although this matroid depth parameter was referred to as branch-depth in~\cite{KarKLM17}.
The main result of this paper is a characterization of contraction$^*$-depth that
corresponds to the definition of tree-depth based on closures of rooted forests.

The definition of tree-depth of a graph $G$ using the closure of rooted forests
can be viewed as first augmenting the graph $G$ with additional edges and
then taking a shallow DFS spanning forest (also known as a Tr\'{e}maux forest)
of the augmented graph.
Here, some intuition can be gained from represented matroids
where the notions of contraction$^*$-depth and a recursively defined parameter contraction-depth,
which was introduced in~\cite{DevKO20} (the formal definition is given in Subsection~\ref{subsec:depth}),
can be shown to intertwine as follows:
every represented matroid $M$,
excluding matroids formed by loops and coloops only,
can be augmented by additional elements in such a way that
the contraction-depth of the augmented (represented) matroid is one plus the contraction$^*$-depth of $M$.
Hence, the contraction$^*$-depth of a represented matroid $M$ can be understood as
the smallest contraction-depth of a representable extension of $M$ minus one.
Our main result asserts that the same statement is true for all matroids:
it is possible to add elements in a way that the added elements
represent the ``hidden spanning forest'' along which an optimal
sequence of contractions can be performed.
We now state the main result formally;
we use $\csd(M)$ to denote the contraction$^*$-depth of a matroid $M$ and
$\cd(M)$ to denote the contraction-depth of $M$.

\begin{theorem}
\label{thm:main}
Let $M$ be a matroid.
The minimum contraction-depth of a matroid $M'$ that contains $M$ as a restriction
is equal to the contraction$^*$-depth of $M$ increased by one,
i.e.,
\[\csd(M)=\min_{M'\sqsupseteq M}\cd(M')-1\]
unless every element of $M$ is either a loop or a coloop and does not consist
solely of loops (when this happens, then $\csd(M)=\cd(M) \in \set{0, 1}$).
\end{theorem}

\noindent We remark that one plus the contraction$^*$-depth of a matroid $M$
is a lower bound on the contraction-depth of $M$ and so
on the contraction-depth of any matroid that contains $M$ (see Theorem~\ref{thm:upper}), and
so Theorem~\ref{thm:main} asserts that this lower bound is attained for some extension of $M$.

We now briefly describe how the paper is structured and survey the main idea of the proof of Theorem~\ref{thm:main}.
In Section~\ref{sec:prelim} we review basic definitions from matroid theory and
depth parameters with a focus on matroid depth parameters.
In Section~\ref{sec:upper} we establish that one plus the contraction$^*$-depth is
a lower bound on the contraction-depth of a matroid, and in Section~\ref{sec:aux}
we prove auxiliary results on matroids needed for the proof of Theorem~\ref{thm:main},
which is given in Section~\ref{sec:main}.
In the proof of Theorem~\ref{thm:main}
we consider an optimal contraction$^*$-decomposition $(T,f)$ of a given matroid $M$,
which we leverage to augment the matroid $M$ by additional elements that
can be used in the recursive procedure described in the definition of contraction-depth.

When deriving the strategy of the proof of Theorem~\ref{thm:main},
it was particularly instructive to understand the case of represented matroids.
It turns out that if $M$ is a represented matroid,
then it is possible to label edges of $T$ with vectors in the linear span of $M$ so that
all elements assigned to some leaf $\ell$ of $T$ lie in the linear span of the edges on the path from $\ell$ to the root.
Inspired by this observation, we introduce additional elements that are associated with the edges of $T$ and
that are most generic in the sense that
they are independent of other elements unless the structure of the contraction$^*$-decomposition forces otherwise.
This idea is captured in the definition of a tamed set given in Subsection~\ref{subsec:tamed}.
In Subsection~\ref{subsec:anal} tamed sets are shown to form a matroid that contains
the original matroid as a restriction.
Finally,
it is shown that the matroid formed by tamed sets has a small contraction-depth in Subsection~\ref{subsec:decomp},
which finishes the proof of Theorem~\ref{thm:main}.

We conclude with a brief discussion of the role of exceptional matroids in the statement of Theorem~\ref{thm:main}
in Section~\ref{sec:concl}.

\section{Preliminaries}
\label{sec:prelim}

In this section we introduce notation used throughout the paper.
Among less standard notation that we use,
we mention that $[k]$ denotes the set of the first $k$ positive integers and
$\overline{X}$ is the complement of the set $X$ (it will always be obvious from the context what
the host set is).
For a general introduction to matroid theory and particularly for more detailed exposition of the notions introduced next,
we refer the reader to the monograph by Oxley~\cite{Oxl11}.

\subsection{Basic definitions from matroid theory}

A \emph{matroid} $M$ is a pair~$(E,\II)$,
where~$\II$ is a non-empty hereditary
(meaning that $\II$ is closed under taking subsets)
collection of subsets of~$E$ that satisfies the \emph{augmentation axiom},
i.e., if $X\in\II$, $X'\in\II$ and $\lvert X\rvert<\lvert X'\rvert$,
then for some $x \in X' \setminus X$ we have $X \cup \set{x} \in \II$.
The set $E$ is the \emph{ground set} of the matroid $M$ and
the sets that belong to~$\II$ are referred to as \emph{independent}.
All matroids considered in this paper have a non-empty ground set and are finite
although our arguments extend to matroids with finite rank.

We often think of a matroid $M$ as its ground set with an additional structure,
so we refer to the elements of the ground set as elements of the matroid $M$,
we write $e\in M$ if $e$ is an element of (the ground set of) the matroid $M$, and
we write $\lvert M\rvert$ for the number of elements of $M$.
The simplest example of a matroid is the \emph{free matroid}, that
is, a matroid such that every subset of its elements is independent.
Another example is the class of vector matroids:
a \emph{vector matroid} is a matroid whose ground set is formed by vectors and
independent sets are precisely sets of linearly independent vectors.

Fix a matroid $M$ for the rest of this subsection.
The \emph{rank} of a set $X$ of elements of $M$,
which is denoted by $\rank_M(X)$, or simply by $\rank(X)$ if $M$ is clear from the context,
is the maximum size of an independent subset of~$X$ (it can be shown using the augmentation axiom that
all maximal independent subsets of $X$ have the same cardinality).
The rank function of any matroid $M$ is \emph{submodular}, meaning that for any $X, Y \subseteq M$
we have
\[
    \rank_M\paren*{X \cup Y} + \rank_M\paren*{X \cap Y} \leq \rank_M(X) + \rank_M(Y) \text{.}
\]
The \emph{rank} of the matroid $M$, which is denoted by $\rank(M)$, is the rank of its ground set.
In the case of vector matroids,
the rank of $X$ is exactly the dimension of the linear space spanned by $X$.

An element $x$ of the matroid $M$ is a \emph{loop} if~$\rank(\set{x})=0$,
an element $x$ is a \emph{coloop}, also called a \emph{bridge}, if~$\rank(M\setminus\set{x})=\rank(M)-1$,  and
two elements $x$ and $x'$ are \emph{parallel} if~$\rank(\set{x})=\rank(\set{x'})=\rank(\set{x,x'})=1$.
A \emph{basis} of the matroid $M$ is an inclusion-wise maximal independent set, and
a \emph{circuit} is an inclusion-wise minimal subset that is not independent.
The \emph{dual matroid}, which is denoted by $M^*$, is the matroid with the same ground set such that
a set $X$ is independent in $M^*$ if and only if $\rank_M(\overline{X})=\rank(M)$
where $\overline{X}$ is the complement of the set $X$ (with respect to the ground set).
In particular, it holds that $\rank_{M^*}(X)=\rank_M(\overline{X})+\card{X} - \rank(M)$ for every set $X$.
We remark that $(M^*)^*=M$ for every matroid $M$.

The \emph{restriction} of the matroid $M$ to a subset $X$ of its elements
is the matroid with the ground set $X$ such that
a set $X'\subseteq X$ is independent in the restriction if and only if $X'$ is independent in $M$;
the restriction of $M$ to $X$ is denoted by $M|X$.
If a matroid $M'$ is a restriction of $M$, we write $M'\sqsubseteq M$.
The matroid obtained from $M$ by \emph{deleting} a set $X$ of the elements of $M$
is the matroid $M|\overline{X}$ and is denoted by $M\setminus X$.
The matroid obtained from $M$ by \emph{contracting} a set $X$, which is denoted by $M/X$,
is the matroid with the ground set $\overline{X}$ such that
a set $X'\subseteq\overline{X}$ is independent in $M/X$ if and only if
$\rank_M(X'\cup X)=\lvert X'\rvert+\rank_M(X)$.
Note that for any subset $X' \subseteq M / X$ we have
\[\rank_{M / X}(X') = \rank_{M}(X' \cup X) - \rank_{M}(X)\text{.}\]
For an element $e$ of the matroid $M$
we write $M\setminus e$ and $M/e$ instead of $M\setminus\{e\}$ and $M/\{e\}$ respectively.
We remark that $(M/e)^*=M^*\setminus e$ and $(M\setminus e)^*=M^*/e$ for every element $e$ of $M$.

We say that the matroid $M$ is \emph{connected} if every two distinct elements of~$M$ are contained in a common circuit;
the property of being contained in a common circuit is transitive~\cite[Proposition 4.1.2]{Oxl11},
that is, if the pair of elements $e$ and $e'$ is contained in a common circuit and
the pair $e'$ and $e''$ is also contained in a common circuit,
then the pair $e$ and $e''$ is also contained in a common circuit.
A \emph{component} of $M$ is an inclusion-wise maximal connected restriction of $M$;
a component is \emph{trivial} if it consists of a single loop, and it is \emph{non-trivial} otherwise.
If $M_1,\ldots,M_k$ are the components of $M$ and $X_i$ is a subset of elements of $M_i$ for each $i\in [k]$,
then the following holds:
\[\rank_M\left(X_1\cup\cdots\cup X_k\right)=\sum_{i\in [k]}\rank_{M_i}(X_i).\]
The converse of this statement is also true, in the sense that
if $X_1, \dots, X_k$ is a partition of the set of elements of $M$ such that
\[\sum_{i \in [k]} \rank(X_i)=\rank(M)\]
then each $X_i$ is a union of components of $M$, see~\cite[Proposition 4.2.1]{Oxl11}.

\subsection{Depth parameters}
\label{subsec:depth}

In this subsection we define matroid depth parameters studied in this paper.
To do so, we need to fix some notation related to \emph{rooted trees},
i.e., trees with a single distinguished vertex referred to as the \emph{root}.
Let $T$ be a rooted tree.
An \emph{ancestor} of a vertex $v$
is any vertex on the path from $v$ to the root (exclusive $v$ itself), and
a \emph{descendant} of $v$ is any vertex $v'$ such that $v$ is an ancestor of $v'$.
The unique neighbor of a non-root vertex $v$ that is its ancestor is the \emph{parent} of $v$ and
any neighbor of $v$ that is its descendant is a \emph{child} of $v$.
A \emph{leaf} is a vertex that has no children;
in particular, the root is a leaf if and only if $\card{V(T)} = 1$.
A \emph{branching vertex} is a vertex that has at least two children and
an \emph{internal branching vertex} is a branching vertex that is not the root.
Note that every internal branching vertex has degree at least three.

With a rooted tree $T$ we can associate a partial order $\preccurlyeq$ on $V(T)$
where $u \preccurlyeq v$ if and only if $u = v$ or $u$ is a descendant of $v$.
Observe that 
the set of elements greater than $v$ in the order $\preccurlyeq$,
i.e., the set of ancestors of $v$,
is linearly ordered by $\preccurlyeq$.
If $e=uv$ is an edge of $T$ and $u\preccurlyeq v$,
then $u$ is the \emph{bottom} vertex of the edge $e$, and
$v$ is its \emph{top} vertex.
For a subset $A \subseteq V(T)$
the \emph{upwards closure} of $A$
is the subtree of $T$ formed by vertices of $A$ and all vertices greater than a vertex of $A$ in $\preccurlyeq$.
Equivalently, the upwards closure of $A$
is the inclusion-wise smallest subtree of $T$ containing the root and all vertices of $A$.
The upwards closure of $A$ is denoted by $T\langle A \rangle$.
If $v$ is a vertex of $T$,
then $T[v]$ is the subtree of $T$ formed by $v$ and all its descendants.
Note that $T[v]$ itself can be viewed as a rooted tree with $v$ being its root.
Finally, if $v$ is a vertex of $T$,
then $\udt{T}{v}$ is the rooted subtree of $T$ formed by $T[v]$ and the path from $v$ to the root (the root of $\udt{T}{v}$ is the root of $T$).
Note that $\udt{T}{v} = T\langle T[v] \rangle$.

The \emph{height} of a rooted tree is the maximum number of vertices on a path from the root to a leaf, and
its \emph{depth} is the maximum number of edges on a path from the root to a leaf.
Note that the height and the depth of a rooted tree always differ by one.
Throughout the paper, we sometimes work with the height of a rooted tree and sometimes with the depth;
in this way, we avoid cumbersome expressions that would otherwise require adding or subtracting one.
A \emph{rooted forest} is a graph such that each component is a rooted tree, and
the \emph{height} of a rooted forest $F$ is the maximum height of a component of $F$.
The \emph{closure~$\cl(F)$} of a rooted forest $F$
is the graph obtained by adding edges from each vertex to all its descendants, and
the \emph{tree-depth~$\td(G)$} of a graph~$G$ is the minimum height of a rooted forest $F$ such that
$\cl(F)$ contains~$G$ as a subgraph.

The \emph{contraction-depth} of a matroid $M$, denoted by $\cd(M)$, is defined recursively as follows~\cite{DevKO20}:
\begin{itemize}
\item If $M$ consists of a single element, then $\cd(M) = 1$.
\item If $M$ is not connected, then $\cd(M)$ is the maximum contraction-depth of a component of $M$.
\item Otherwise $\cd(M)=1+\min\limits_{e\in M}\cd(M/e)$,
      i.e., $\cd(M)$ is one plus the minimum contraction-depth of $M/e$ where the minimum is taken over all elements $e$ of $M$.
\end{itemize}      
Similarly,
the \emph{deletion-depth} of a matroid $M$, denoted by $\dd(M)$, is defined recursively as follows~\cite{DevKO20}:
\begin{itemize}
\item If $M$ consists of a single element, then $\dd(M) = 1$.
\item If $M$ is not connected, then $\dd(M)$ is the maximum deletion-depth of a component of $M$.
\item Otherwise $\dd(M)=1+\min\limits_{e\in M}\dd(M\setminus e)$,
      i.e., $\dd(M)$ is one plus the minimum deletion-depth of $M\setminus e$ where the minimum is taken over all elements $e$ of $M$.
\end{itemize}      
It can be shown that the contraction-depth and deletion-depth of a matroid are dual notions,
that is, $\cd(M)=\dd(M^*)$ for every matroid $M$.

We now define the notion of the contraction$^*$-depth of a matroid,
which was introduced under the name of branch-depth in~\cite{KarKLM17}.
Since there is a competing notion of branch-depth~\cite{DevKO20},
we follow the terminology used in~\cite{BriKKPS24}, also see the conference version~\cite{BriKKPS22}.
A \emph{contraction$^*$-decomposition} of a matroid $M$ is a pair $\TT = (T, f)$
where $T$ is a rooted tree with $\rank(M)$ edges, and
$f$ is a mapping from the elements of $M$ to the leaves of $T$ such that
the following holds for any subset $X$ of elements of $M$:
\[|E(T \langle f(X) \rangle)| \geq \rank(X)\text{,}\]
that is, the number of edges of the upwards closure of the set of leaves to which the elements of $X$ are mapped by $f$
is at least the rank of $X$.
See Figure~\ref{fig:decomp} for an example.
If $(T,f)$ is a contraction$^*$-decomposition and $v$ is a vertex of $T$,
then $T(v)$ denotes the set of all elements of $M$ mapped to the leaves of the subtree $T[v]$ rooted at $v$.

\begin{figure}
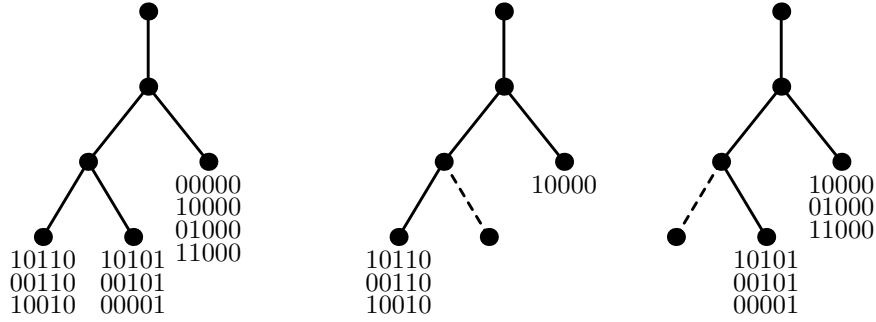

\begin{center}
\epsfbox{cdepth-3.mps}
\hskip 8ex
\epsfbox{cdepth-31.mps}
\hskip 4ex
\epsfbox{cdepth-32.mps}
\end{center}
\caption{An example of a contraction$^*$-decomposition $(T,f)$ of a matroid $M$ (in the left);
         the matroid $M$ is a $10$-element vector matroid of rank five represented over the binary field.
	 The vectors assigned by the function $f$ to each of the three leaves are listed as row vectors by the leaves.
	 The subtrees $T\langle f(X)\rangle$ for $X=\{10000, 10110, 00110, 10010\}$ and 
	 for $X=\{10000, 01000, 11000, 10101, 00101, 00001\}$
	 are depicted in the middle and in the right.}
\label{fig:decomp}
\end{figure}

The \emph{depth} of a contraction$^*$-decomposition $(T,f)$ is the depth of the tree $T$, and
the \emph{contraction$^*$-depth} of a matroid $M$
is the minimum depth of a contraction$^*$-decomposition of $M$.
It can be shown that contraction$^*$-depth of matroids
is non-increasing with respect to deleting and contracting elements
(see Proposition 3.3 in~\cite{KarKLM17} for a proof),
while contraction-depth is non-increasing only with respect to contracting elements
and might increase when deleting elements.

There is an alternative definition of the contraction$^*$-depth for representable matroids,
which can be viewed as a generalization of the contraction-depth.
In particular,
the contraction$^*$-depth of a vector matroid $M$ can be equivalently defined as follows~\cite{ChaCKKP20,ChaCKKP22}:
\begin{itemize}
\item If $M$ consists of a single element, then $\csd(M)=\rank(M)$.
\item If $M$ is not connected, then $\csd(M)$ is the maximum contraction$^*$-depth of a component of $M$.
\item Otherwise $\csd(M) = 1 + \min\limits_{L} \csd(M / L)$
      where $L$ is any one-dimensional subspace of the linear span $K$ of $M$ and
      $M/L$ is the vector matroid formed by the vectors of $M$ projected to the quotient space $K/L$ (see~\cite{Hal93}
      for the definition of a quotient space).
\end{itemize}
So, contraction$^*$-depth generalizes contraction-depth for vector matroids, in the sense that
the contraction-depth of a vector matroid $M$ decreased by one would coincide
with its contraction$^*$-depth had the definition of the latter required that the subspace $L$ is spanned by an element of $M$ (with
the exception of the exceptional matroids described in Theorem~\ref{thm:main}).

Contraction-depth and contraction$^*$-depth are functionally equivalent;
the following holds for every matroid $M$~\cite{KarKLM17}:
\[\csd(M)\le\cd(M)\leq 4^{\csd(M)}+1\text{.}\]
Note that
when $M$ is a vector matroid,
the first inequality follows easily from the alternative definition of
contraction$^*$-depth presented above.
For completeness, we provide a proof of the first inequality in Section~\ref{sec:upper}
where we also show that it is actually tight only for non-zero rank matroids that contain loops and coloops only.

\section{Upper bound on contraction$^*$-depth}
\label{sec:upper}

In this section we prove that $\csd(M)\le\cd(M)-1$ for a matroid $M$
unless $M$ consists of loops and coloops only.
To do so, we need the following lemma.

\begin{lemma}
\label{lm:bridge}
Let $M$ be a matroid and $X$ an arbitrary set of elements of $M$.
If an element $e\in M\setminus X$ is not a coloop of $M$,
then $e$ is also not a coloop in $M/X$.
\end{lemma}

\begin{proof}
Recall that an element $e$ is a coloop in $M$ if and only if $e$ a loop in the dual matroid $M^*$.
Since the element $e$ is not a loop in $M^*$, $e$ is also not a loop in $M^*\setminus X$ and so in $(M/X)^*$.
It follows that $e$ is not a coloop in $M/X$.
\end{proof}

We are now ready to prove the upper bound on contraction$^*$-depth in terms of contra\-ction-depth.

\begin{theorem}
\label{thm:upper}
Let $M$ be a matroid.
Then $\csd(M)\le\cd(M)-1$
unless every element of $M$ is either a loop or a coloop and $M$ does not consist solely of loops.
\end{theorem}

\begin{proof}
The proof proceeds by induction on the number of elements of $M$.
We first analyze the case when $\rank (M)=0$,
which also includes the case when $M$ is empty.
If $\rank (M)=0$, then the matroid $M$ consists of loops only.
It follows that $\cd(M)=1$ (unless it is empty) and $\csd(M)=0$ (the unique contraction$^*$-decomposition of $M$
is the rooted tree consisting of the root only and all loop elements of $M$ assigned to the root).

Suppose that $\rank (M)=1$.
Since $M$ has an element that is not a coloop,
$M$ consists of two or more parallel elements and some (possibly zero) number of loops.
It follows that $\cd(M)=2$ (the contraction of any non-loop element results
in a non-empty matroid consisting of loops only and so of contraction-depth $1$).
The contraction$^*$-depth of $M$ is equal to one and
the unique contraction$^*$-decomposition of $M$ is the rooted tree with a single edge and
all elements of $M$ mapped to its only leaf.

Suppose that $\rank (M)\geq 2$.
If $M$ is not connected, then let $M_1,\ldots,M_k$ be the components of $M$.
We may assume that the first $\ell>0$ components do not consist of a loop or a coloop only
while the remaining $k-\ell$ ones do (note that $\ell$ can be equal to $k$).
Since the contraction-depth of the components $M_{\ell+1},\ldots,M_k$ is equal to one,
the contraction-depth of $M$ is the maximum of the contraction-depths of $M_1,\ldots,M_{\ell}$.
By induction,
there exist contraction$^*$-decompositions $\TT_1,\ldots,\TT_{\ell}$ of the matroids $M_1,\ldots,M_{\ell}$ respectively,
such that the depth of $\TT_i=(T_i,f_i)$ is at most $\cd(M_i)-1$ for each $i\in [\ell]$.
Let $T$ be the rooted tree obtained by identifying the roots of $T_1,\ldots,T_{\ell}$ and
adding a new leaf adjacent to the root for every component $M_{\ell+1},\ldots,M_k$ that
consists of a coloop only.
We define the function $f$ from the elements of $M$ to the leaves of $T$ as follows.
If $e$ is an element of $M_i$ for $i\in [\ell]$, then $f(e)=f_i(e)$.
If $e$ is a coloop, then $e$ is mapped by $f$ to the newly added leaf corresponding to the component consisting of $e$.
Finally, if $e$ is a loop, then $f(e)$ is an arbitrary leaf of $T$.
Observe that $(T,f)$ is a contraction$^*$-decomposition of $M$.
The depth of $T$ is the maximum depth of $T_i$ where $i \in [\ell]$,
thus it is at most \[\max_{i\in [\ell]}\cd(M_i)-1 = \cd(M) - 1.\]
It follows that $\csd(M)\le\cd(M)-1$.

If $\rank (M)\geq 2$ and the matroid $M$ is connected, there exists an element $e$ such that $\cd(M/e)=\cd(M)-1$.
By Lemma~\ref{lm:bridge}, the matroid $M/e$ has no coloops (it may have loops however).
In particular, it is possible to apply induction.
Let $(T,f)$ be a contraction$^*$-decomposition of $M/e$ with depth
at most $\cd(M/e)-1=\cd(M)-2$.
We obtain a contraction$^*$-decomposition of $M$ by creating a new root $v_0$ and
making $v_0$ adjacent to the root of $T$;
the element $e$ is mapped to an arbitrary leaf of $T$.
The obtained rooted tree is a contraction$^*$-decomposition of $M$
since $\rank_M (X)\le\rank_{M/e}(X\setminus\{e\})+1$ and
any path from the root to any leaf contains the edge between $v_0$ and the root of $T$.
Since the depth of the obtained contraction$^*$-decomposition of $M$ is at most $\cd(M)-1$,
we conclude that $\csd(M)\le\cd(M)-1$.
\end{proof}

\section{Auxiliary results}
\label{sec:aux}

In this section
we prove several auxiliary results used in the proof of our main theorem.
We start with the following lemma,
which (informally) says that
if the root of a contraction$^*$-decomposition of a matroid $M$ has degree at least two,
then $M$ is not connected and
the elements of each component of $M$ are mapped to the leaves of the same subtree of a child of the root.
We refer to Figure~\ref{fig:lm:rank} for illustration of the notation used in the statement of the lemma.

\begin{figure}[H]
\begin{center}
\epsfbox{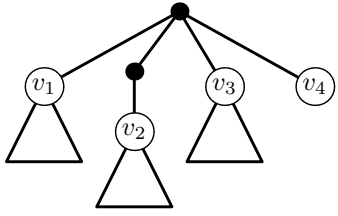}
\end{center}
\caption{Illustration of the notation used in the statement of Lemma~\ref{lm:rank}.}
\label{fig:lm:rank}
\end{figure}

\begin{lemma}
\label{lm:rank}
Let $(T, f)$ be a contraction$^*$-decomposition of a matroid $M$ and
let $v_1,\dots,v_k$ be the $\preccurlyeq$-maximal descendants of the root that are branching vertices or leaves.
Then each set $T(v_i)$, $i\in [k]$, is a union of components of the matroid $M$.
In particular, the following holds for every subset $X$ of elements of $M$:
\[\rank_M(X)=\sum_{i\in [k]}\rank_M\left(X\cap T(v_i)\right).\]
\end{lemma}

\begin{proof}
Fix a matroid $M$ and a contraction$^*$-decomposition $(T,f)$ of $M$.
Let $v_1,\ldots,v_k$ be the vertices of $T$ as in the statement of the lemma and
let $t_i$ be the number of edges of the subtree $\udt{T}{v_i}$, $i\in [k]$.
Note that the tree $T$ has exactly $t_1+\cdots+t_k$ edges.

We will show that $\rank (T(v_i))=t_i$ for every $i \in [k]$.
The definition of a contraction$^*$-decomposition yields that $\rank (T(v_i))\le t_i$ for every $i\in [k]$.
By symmetry, we may assume for contradiction that $\rank (T(v_1))<t_1$.
It follows that $\rank (M\setminus T(v_1)) > t_2+\cdots+t_k$.
However,
the number of edges from the leaves of $T[v_2],\ldots,T[v_k]$ to the root is only $t_2+\cdots+t_k$,
which is impossible since $(T,f)$ is a contraction$^*$-decomposition of $M$.
We conclude that $\rank (T(v_i))=t_i$ for every $i \in [k]$.

Since the rank of the matroid $M$ is $t_1+\cdots+t_k$,
the sets $T(v_1),\ldots,T(v_k)$ partition the elements of $M$, and
$\rank\left(T(v_i)\right)=t_i$ for every $i\in [k]$,
it must be the case that
each set $T(v_i)$ is a union of components of $M$ and so the lemma follows.
\end{proof}

To prove the final lemma of this section, we need the following auxiliary lemma.

\begin{lemma}
\label{lm:onechange}
Let $M$ be a matroid, and
let $X$ and $Z_1\supseteq Z_2\supseteq\cdots\supseteq Z_k$ be subsets of elements of $M$.
Then for any element $e\not\in X\cup Z_1$
there exists $i_0\in [k+1]$ such that
\[
\rank_{M/Z_i}((X\setminus Z_i)\cup\set{e})=
\begin{cases}
\rank_{M/Z_i} (X\setminus Z_i) & \mbox{if $i<i_0$, and}\\
\rank_{M/Z_i} (X\setminus Z_i) + 1 & \mbox{if $i\geq i_0$}
\end{cases}
\]
for all $i\in [k]$.
\end{lemma}

\begin{proof}
Recall that the \emph{span} of a subset $A$ of the elements of a matroid is the inclusion-wise maximal subset of $A$ with the same rank, and
if $A\subseteq B$, then the span of $A$ is a subset of the span of $B$.
Observe that an element $x$ is contained in the span of $A\cup B$ if and only if $\rank_M A\cup B=\rank_M A\cup B\cup\{x\}$,
which is equivalent to $\rank_{M/B}A=\rank_{M/B}A\cup\{x\}$.

If the element $e$ is not contained in the span of $X\cup Z_1$,
then $e$ is not contained in the span of $X\cup Z_i$ for all $i\in [k]$ and
so $\rank_{M/Z_i}((X\setminus Z_i)\cup\set{e})=\rank_{M/Z_i} (X\setminus Z_i) + 1$ for all $i\in [k]$;
in particular, the statement of the lemma holds with $i_0=1$.
If the element $e$ is contained in the span of $X\cup Z_1$,
set $i_0=j_0+1$ where $j_0\in [k]$ is the largest integer such that $e$ is contained in the span $X\cup Z_{j_0}$.
Since it holds that $\rank_{M/Z_i}((X\setminus Z_i)\cup\set{e})=\rank_{M/Z_i} (X\setminus Z_i)$ for $i\in [j_0]$ and
$\rank_{M/Z_i}((X\setminus Z_i)\cup\set{e})=\rank_{M/Z_i} (X\setminus Z_i) + 1$ for $i\in [k]\setminus [j_0]$,
the statement of the lemma follows.
\end{proof}

We conclude this section with the next lemma concerning ranks
after contracting a set defined by a contraction$^*$-decomposition;
we refer to Figure~\ref{fig:lm:edges} for illustration of the notation.

\begin{figure}[h]
\begin{center}
\epsfbox{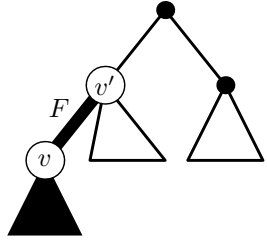}
\end{center}
\caption{Illustration of the notation used in the statement of Lemma~\ref{lm:edges}.}
\label{fig:lm:edges}
\end{figure}

\begin{lemma}
\label{lm:edges}
Let $M$ be a matroid and $(T,f)$ a contraction$^*$-decomposition of $M$.
Let $v$ be a leaf or an internal branching vertex and
$v'$ its $\preccurlyeq$-minimal ancestor that is a branching vertex or the root.
Let $F$ be the union of the set of edges contained inside $T[v]$
and the set of edges on the path from $v$ to $v'$.
Then the size of the set $F$ is at most $\rank_{M/\overline{T(v)}} \paren*{T(v)}$.
\end{lemma}

\begin{proof}
Observe that $\rank (M)=\rank_M \paren*{\overline{T(v)}} + \rank_{M/\overline{T(v)}} \paren*{T(v)}$.
Since $(T,f)$ is a contraction$^*$-decomposition of $M$,
the number of edges not contained in $F$ must be at least $\rank_M \paren*{\overline{T(v)}}$.
Since the total number of edges of $T$ is $\rank (M)$,
it follows that the size of $F$ is at most $\rank (M)-\rank_M \paren*{\overline{T(v)}} = \rank_{M/\overline{T(v)}} \paren*{T(v)}$.
\end{proof}

\section{Main result}
\label{sec:main}

This section is devoted to the proof of our main result.
We start with an informal description of the proof idea.
Our goal is to use a contraction$^*$-decomposition $\TT=(T,f)$ of a matroid $M$ to construct a suitable extension of $M$.
Suppose that $v$ is an internal branching vertex of $T$.
Intuitively,
the length of the path from $v$ to the root corresponds to the ``rank'' of the ``intersection'' of the matroids formed by the elements assigned to different subtrees of $v$.
This intuition manifests in the case of vector matroids as follows (see~\cite{ChaCKKP22}).
If $\TT=(T,f)$ is a contraction$^*$-decomposition of a vector matroid $M$,
it is possible to label the edges of $T$ with vectors in a way that
any subset $X$ of the elements of $M$ is contained in the linear hull of the labels of the edges of $\langle f(X) \rangle$
(note that this implies that $\rank(X)$ is at most the number of the edges of $\langle f(X) \rangle$).
In particular, in the case of vector matroids,
adding the vectors that are labels of the edges yield the sought extension of $M$.
In the general case, we will also introduce a new element for each edge of the tree $T$ and
our aim will be to do so in a way that
the new elements behave as ``generic'' elements of ``intersections'' of the matroids formed by corresponding subtrees,
i.e., they do not create new dependencies unless forced by the submodularity of the rank function.

\subsection{Tamed sets}
\label{subsec:tamed}

In this subsection, we define the notion of tamed sets.
In the following subsections we show that tamed sets form a matroid which contains the given matroid as a restriction, and
we then analyze the contraction-depth of the matroid formed by tamed sets.
For the rest of this subsection,
fix a matroid $M$ and a contraction$^*$-decomposition $\TT = (T, f)$ of $M$.
Let $X_0$ be the set containing all elements of $M$ and all edges of $T$, i.e., $X_0=M\cup E(T)$;
note that $X_0$ has $\lvert M\rvert+\rank (M)$ elements.
The rest of the subsection is devoted to defining when a set $X\subseteq X_0$ is $\TT$-tamed.

\textbf{Token assignment.}
For a set $X\subseteq X_0$,
we define an assignment of tokens to the vertices of the tree $T$,
which is next referred as to the \emph{token assignment with respect to the set $X$}.
Every vertex that is neither a leaf nor an internal branching vertex receives zero tokens.
Let $v$ be a vertex that is either a leaf or an internal branching vertex, and
let $v'$ be the $\preccurlyeq$-minimal ancestor of $v$ that is a branching vertex or the root.
Observe that all internal vertices of the path from $v$ to $v'$ have degree two, and
let $P$ be the set of all edges on the path from $v$ to $v'$ in $T$.
\begin{itemize}
\item If $v$ is a leaf,
      then the number of tokens assigned to the leaf $v$ is equal to 
      \begin{equation}
      \lvert X\cap P\rvert+\rank_{M/\overline{T(v)}} (X\cap T(v))\text{.}\label{eq:leaf}
      \end{equation}
\item If $v$ is an internal branching vertex,
      then we proceed as follows.
      Let $v_1,\ldots,v_k$ be all $\preccurlyeq$-maximal descendants of $v$ that are either leaves or branching vertices.
      The number of tokens assigned to $v$ is equal to
      \begin{equation}
      \lvert X\cap P\rvert+\rank_{M/\overline{T(v)}} (X\cap T(v)) -\sum_{i\in [k]}\rank_{M/\overline{T(v_i)}} (X\cap T(v_i))\text{.}\label{eq:branch}
      \end{equation}
\end{itemize}
This concludes the definition of the token assignment with respect to $X$.

In the next lemma, we show that the notion of a token assignment is well-defined in the sense that
the number of tokens assigned to each vertex is non-negative,
i.e., the expression \eqref{eq:branch} is non-negative for every $X\subseteq X_0$ and every internal branching vertex $v$.

\begin{lemma}
\label{lm:nonnegative}
Let $X$ be an independent set in the matroid $M$. Let $v$ be an internal branching vertex of $T$,
and let $v_1, \dots, v_k$ be all $\preccurlyeq$-maximal descendants of $v$ that are either
branching vertices or leaves. Then
\[
    \rank_{M / \overline{T(v)}}(X \cap T(v)) \geq \sum\limits_{i \in [k]}\rank_{M / \overline{T(v_i)}}(X \cap T(v_i))\text{.}
\]
\end{lemma}
\begin{proof}
For every $i \in [k]$,
let $Y_i$ be a subset of $X \cap T(v_i)$ of size $\rank_{M / \overline{T(v_i)}}(X \cap T(v_i))$
that is independent in $M / \overline{T(v_i)}$.
Set $Y = Y_1 \cup \dots \cup Y_k$.
Since $Y$ is a subset of $X \cap T(v)$,
the lemma will be proved once we establish that $Y$ is independent in $M / \overline{T(v)}$.

We prove the following claim by induction on $\ell$:

\medskip
\emph{For every $\ell \in [k]$, the set $Y_1 \cup \dots \cup Y_\ell$ is independent in $M / \overline{T(v)}$}.
\medskip

\noindent The base case of the induction holds as $Y_1$ is independent in $M / \overline{T(v_1)}$ and $\overline{T(v)} \subseteq \overline{T(v_1)}$,
therefore $Y_1$ is independent in $M / \overline{T(v)}$.
Suppose we have shown that $Y_1 \cup \dots \cup Y_\ell$ is independent in $M / \overline{T(v)}$ for some $\ell < k$.
Let $Z$ be an independent superset of $Y_1 \cup \dots \cup Y_\ell$ contained in $T(v) \setminus T(v_{\ell + 1})$ of size
$\rank_{M / \overline{T(v)}}(T(v) \setminus T(v_{\ell + 1}))$.
As $Y_{\ell + 1}$ is independent in $M / \overline{T(v_{\ell + 1})} = (M / \overline{T(v)}) / (T(v) \setminus T(v_{\ell + 1}))$,
it must be the case that $Z \cup Y_{\ell + 1}$ is independent in $M / \overline{T(v)}$.
In particular, $Y_1 \cup \dots \cup Y_{\ell + 1}$ is independent in $M / \overline{T(v)}$, and the claim follows.
\end{proof}

Before proceeding further, we compute the total number of tokens assigned to the vertices of $T$.
To do so, we will need the following lemma.
Recall that the matroid $M$ and the contraction$^*$-decomposition $(T,f)$ of $M$ are fixed.

\begin{lemma}
\label{lm:tokens-aux}
Let $X$ be any subset of the set $M\cup E(T)$ such that $X\cap M$ is independent, and
consider the token assignment with respect to the set $X$.
Let $v$ be a a leaf or an internal branching vertex of $T$, and
let $F$ be the union of the set of all edges contained in $T[v]$ and
the set of edges on the path from $v$ to the $\preccurlyeq$-minimal ancestor of $v$ that is a branching vertex or the root.
Then the total number of tokens assigned to all vertices of the subtree $T[v]$
is equal to $\lvert X\cap F\rvert+\rank_{M/\overline{T(v)}}(X\cap T(v))$.
\end{lemma}

\begin{proof}
If $v$ is a leaf, then the quantity in the statement of the lemma matches \eqref{eq:leaf}.
If $v$ is not a leaf,
then let $v_1,\dots,v_k$ be the $\preccurlyeq$-maximal descendants of $v$ that are branching vertices or leaves.
For every $i\in [k]$, let $F_i$ be the union of the set of edges contained in $T[v_i]$
and the set of edges on the path from $v_i$ to $v$.
Note that the set $F$ from the statement of the lemma is the union of the sets $F_i$ for $i\in [k]$ and
the set of edges of the path from $v$ to its $\preccurlyeq$-minimal ancestor that is a branching vertex or the root.
It follows that $F\setminus\bigcup_{i\in [k]} F_i$
is the set of edges of the path from $v$ to its $\preccurlyeq$-minimal ancestor that is a branching vertex or the root.
Using \eqref{eq:branch},
we obtain that the number of tokens assigned to $v$ itself is equal to
\[
\left\lvert X\cap\left(F\setminus\bigcup_{i\in [k]} F_i\right)\right\rvert+\rank_{M/\overline{T(v)}} (X\cap T(v)) - \sum_{i\in [k]}\rank_{M/\overline{T(v_i)}} (X\cap T(v_i))\text{.}
\]
By induction, it holds that the number of tokens assigned to the vertices of $T[v_i]$ is equal to
\[\left\lvert X\cap F_i\right\rvert+\rank_{M/\overline{T(v_i)}} (X\cap T(v_i))\]
for every $i\in [k]$.
Since the sets $F_1,\ldots,F_k$ are disjoint,
we obtain that the number of tokens assigned to the vertices of $T[v]$ is equal to
\[\left\lvert X\cap F\right\rvert+\rank_{M/\overline{T(v)}} (X\cap T(v))\]
as desired.
\end{proof}

We are now ready to compute the number of tokens assigned with respect to a set $X$.

\begin{lemma}
\label{lm:tokens}
Let $X$ be any subset of the set $M\cup E(T)$ such that $X\cap M$ is independent.
Then the token assignment with with respect to the set $X$ contains exactly $\lvert X\rvert$ tokens.
\end{lemma}

\begin{proof}
Let $v_1,\dots,v_k$ be $\preccurlyeq$-maximal descendants of the root that are branching vertices or leaves, and
let $F_i$ be the set of edges contained in $\udt{T}{v_i}$.
If $k=1$, then Lemma~\ref{lm:tokens-aux} implies that the number of tokens assigned to the vertices of $T$
(note that $T(v_1)=M$ and $F_1$ are all edges of $T$) is equal to
\[\lvert X\cap F_1\rvert+\rank_M (X\cap T(v_1))=\lvert X\cap E(T)\rvert+\lvert X\cap M\rvert=\lvert X\rvert\text{.}\]
Hence, the number of tokens assigned to the vertices of $T$ is equal to $\lvert X\rvert$.

If $k\geq 2$, then Lemma~\ref{lm:tokens-aux} implies that the number of tokens assigned to the vertices of $T$
is equal to
\begin{equation}
\sum_{i\in [k]}\paren*{\left\lvert X\cap F_i\right\rvert+\rank_{M/\overline{T(v_i)}} (X\cap T(v_i))}\text{.}\label{eq:tokens1}
\end{equation}
Note that each $T(v_i)$ is a union of components of the matroid $M$ by Lemma~\ref{lm:rank}.
It follows that
$\rank_{M/\overline{T(v_i)}} (X\cap T(v_i))=\rank_M (X\cap T(v_i))$ for every $i\in [k]$.
Using that $X\cap M$ is independent in $M$ and
that the sets $F_1, \dots, F_k$ partition the edge set of $T$,
we obtain that \eqref{eq:tokens1} is equal to
\begin{align*}
\sum_{i \in [k]}\paren*{\card*{X \cap F_i} + \rank_M (X \cap T(v_i))}
    &= \sum_{i \in [k]} \card{X \cap F_i} + \sum_{i \in [k]}\rank_M (X \cap T(v_i))\\
    &= \card*{X \cap \paren*{F_1 \cup \dots \cup F_k}} + \rank_M(X \cap M)\\
    &= \card*{X \cap E(T)} + \card{X \cap M} = \card{X}\text{.}
\end{align*}
We conclude that the number of tokens assigned to the vertices of $T$ is equal to $\lvert X\rvert$.
\end{proof}

\textbf{Token distribution.}
Once the tokens have been assigned to the leaves and internal branching vertices of $T$,
they are distributed within the tree $T$ in a bottom up fashion using a procedure that we now describe
(also see Figure~\ref{fig:distribution} for an example).
Vertices that have already sent (some of) their tokens up will be \emph{marked}.
At the beginning no vertex is marked.
As long as there is a non-root vertex $v$ such that all of its descendants are marked we do the following.
Let $k$ be the number of tokens held by $v$.
If $k = 0$, then $v$ becomes marked (and sends no tokens as it has no tokens).
If $k > 0$, then $v$ sends $k - 1$ tokens to its parent and $v$ then becomes marked.
Observe that the procedure terminates when all vertices other than the root are marked.

\begin{figure}
\begin{center}
\epsfbox{cdepth-7.mps}
\end{center}
\caption{Illustration of the steps of the procedure for token distribution.
         Marked vertices are denoted by bold circles.}
\label{fig:distribution}
\end{figure}

\textbf{Definition of tamed sets.}
A set $X\subseteq X_0$ is \emph{$\TT$-tamed}
if $X\cap M$ is independent and
the procedure for the token distribution starting with the token assignment with respect to $X$
yields an assignment such that the root has zero tokens at the end of the procedure.
The set of $\TT$-tamed subsets will be denoted by $\II_M^{\TT}$;
if the contraction$^*$-decomposition $\TT$ is clear from the context, we will simply write $\II_M$.

Observe that if the matroid $M$ is given by the independence oracle or by the rank oracle,
then it can be decided in polynomial time whether a subset of $M\cup E(T)$ is $\TT$-tamed.
Also observe that the set $E(T)$, i.e., the set formed by all edges of the tree $T$, is always $\TT$-tamed.
Indeed, if $v$ is a vertex that is either a leaf or an internal branching vertex, and
$v'$ is the $\preccurlyeq$-minimal ancestor of $v$ that is a branching vertex or the root,
the tokens assigned to $v$ in the token assignment with respect to $E(T)$
are distributed exactly to the bottom vertices of the edges of the path from $v$ to $v'$.

\subsection{Analysis}
\label{subsec:anal}

The goal of this subsection is to show that
if $M$ is a matroid and $\TT$ is a contraction$^*$-decomposition of $M$,
then the collection of all $\TT$-tamed sets is hereditary and satisfies the augmentation axiom,
i.e., the $\TT$-tamed sets form a collection of independent sets of a matroid.

We start with showing that the collection of all $\TT$-tamed sets is hereditary.

\begin{lemma}
\label{lm:hereditary}
Let $M$ be a matroid and $\TT=(T, f)$ a contraction$^*$-decomposition of $M$.
Then the collection $\II_M^\TT$ of all $\TT$-tamed sets is hereditary.
\end{lemma}

\begin{proof}
Consider a non-empty set $X$ contained in $\II_M^\TT$ and an element $e\in X$.
Our goal is to show that $X\setminus\set{e}$ is also $\TT$-tamed.

Suppose that $e$ is an edge of the tree $T$ and
let $u$ be the bottom vertex of $e$. If $u$ is a leaf or a branching vertex, then we set $v = u$.
Otherwise we define $v$ to be the $\preccurlyeq$-maximal descendant of $u$ that is either a leaf or a branching vertex
(note that the vertex $v$ is uniquely defined).
Observe that every vertex of $T$ except for $v$ is assigned the same number of tokens
with respect to the set $X$ and with respect to the set $X\setminus\{e\}$, and
the vertex $v$ is assigned exactly one token less.
It follows that
if the root of $T$ gets no tokens
during the procedure for the token distribution starting with the token assignment with respect to $X$,
then the root of $T$ gets no tokens 
during the procedure for the token distribution starting with the token assignment with respect to $X\setminus\{e\}$.
In particular, the set $X\setminus\set{e}$ is $\TT$-tamed.

It remains to deal with the case when $e$ is an element of the matroid $M$.
Let $v_0$ be the leaf such that $f(e) = v_0$, and
let $v_1,\ldots,v_k$ be all internal branching vertices on the path from $v_0$ to the root
listed in the ascending order by $\preccurlyeq$.
We apply Lemma~\ref{lm:onechange} with the matroid $M$,
the set $(X\setminus\{e\})\cap M$ and
$\overline{T(v_0)}\supseteq\overline{T(v_1)}\supseteq\cdots\supseteq\overline{T(v_k)}$.
The lemma yields that there exists $i_0\in\set{0}\cup [k+1]$ such that
\[
\rank_{M/\overline{T(v_i)}}(X\cap T(v_i))=
\begin{cases}
\rank_{M/\overline{T(v_i)}} ((X\setminus\{e\})\cap T(v_i)) & \mbox{if $i<i_0$, and}\\
\rank_{M/\overline{T(v_i)}} ((X\setminus\{e\})\cap T(v_i))+1 & \mbox{if $i\geq i_0$.}
\end{cases}
\]
Note that $i_0\leq k$, as otherwise we would have $i_0=k+1$,
which implies that the number of tokens assigned with respect to the set $X\setminus\{e\}$ and
with respect to the set $X$ would be the same, which is impossible by Lemma~\ref{lm:tokens}.
It follows that the vertex $v_{i_0}$ is assigned one token less
with respect to the set $X\setminus\{e\}$ compared
to the number of tokens assigned with respect to the set $X$;
all other vertices are assigned the same number of tokens.
We conclude that
if the root of $T$ gets no tokens during the distribution phase with respect to the set $X$,
then the root of $T$ gets no tokens during the distribution phase with respect to the set $X\setminus\{e\}$ and
so the set $X\setminus\set{e}$ is $\TT$-tamed.
\end{proof}

We are now ready to prove that the collection of $\TT$-tamed sets is a collection of independent sets of a matroid.

\begin{theorem}
\label{thm:matroid}
Let $M$ be a matroid and $\TT = (T,f)$ a contraction$^*$-decomposition of $M$.
Then the collection $\II_M^T$ is non-empty, hereditary and satisfies the augmentation axiom.
\end{theorem}

\begin{proof}
The collection $\II_M^\TT$ is non-empty.
Indeed, there are no tokens assigned with respect to the empty set, and
so the empty set is $\TT$-tamed.
By~\Cref{lm:hereditary}, the collection $\II_M^\TT$ is hereditary.
Thus, we need to prove that the collection $\II_M^T$ satisfies the augmentation axiom.
Consider $\TT$-tamed sets $X$ and $X'$ such that $\lvert X\rvert<\lvert X'\rvert$.
We will establish that there exists an element $e\in X' \setminus X$ such that $X\cup\set{e}$ is $\TT$-tamed.

For a vertex $v$ of the tree $T$ that is either a leaf or an internal branching vertex,
we define $d(v)$ to be the number of tokens kept by vertices on the path from $v$ (including $v$)
to the $\preccurlyeq$-minimal ancestor of the vertex $v$ that is a branching vertex or the root
(excluding this ancestor)
during the procedure for the token distribution starting with the token assignment with respect to $X$.
Analogously, we define $d'(v)$ for the token assignment with respect to $X'$.
Note that the sum of all $d(v)$ (the sum is taken over all leaves and internal branching vertices)
is the total number of tokens, which, by \Cref{lm:tokens}, equals $\card{X}$.
Similarly, the sum of all $d'(v)$ is $\lvert X'\rvert$.
Let $v$ be a $\preccurlyeq$-minimal vertex such that $d(v)<d'(v)$, and
let $P$ be the set of edges of the path from $v$ to the $\preccurlyeq$-minimal ancestor of $v$
that is a branching vertex or the root;
see Figure~\ref{fig:thm:matroid} for illustration of the notation.
Since $d'(v)$ cannot exceed $\lvert P\rvert$,
$d(v)$ is strictly smaller than $\lvert P\rvert$ and
so at least one of the bottom vertices of the edges in $P$
kept no tokens during the procedure for the token distribution starting with the token assignment with respect to $X$.

\begin{figure}[H]
\begin{center}
\epsfbox{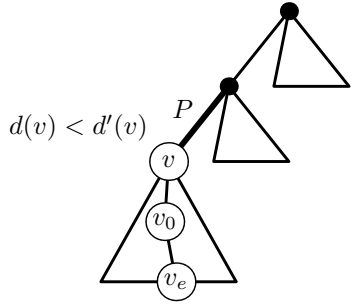}
\end{center}
\caption{Illustration of the notation used in the statement of Theorem~\ref{thm:matroid}.}
\label{fig:thm:matroid}
\end{figure}

If the set $X'$ contains an edge $e\in E(T[v])\cup P$ that is not contained in $X$,
then the set $X\cup\{e\}$ is $\TT$-tamed and so the sought element has been identified.
Indeed, the number of tokens assigned to the vertices increases by one exactly for one vertex of $T[v]$, and
the additional token is either kept by a vertex of $T[v]$ or by one of the bottom vertices of the edges in $P$.

If there is no such edge,
i.e., every edge of $E(T[v])\cup P$ contained in $X'$ is also contained in $X$,
Lemma~\ref{lm:tokens-aux} yields that
\begin{equation}
\rank_{M/\overline{T(v)}} (X\cap T(v)) <\rank_{M/\overline{T(v)}} (X'\cap T(v))\text{.}
\label{eqn:ranksinequality}
\end{equation}
Let $e$ be any element of $X'\cap T(v)$ such that $(X\cap T(v))\cup\set{e}$
has larger rank than the set $X\cap T(v)$ in the matroid $M/\overline{T(v)}$;
the element $e$ exists by \eqref{eqn:ranksinequality}.
In the rest of the proof, we will show that the set $X\cup\set{e}$ is $\TT$-tamed.

Since the set $(X\cap T(v))\cup\set{e}$ is independent in the matroid $M/\overline{T(v)}$,
the set $(X\cap T(v))\cup\set{e}\cup Y$ is independent for any set $Y\subseteq\overline{T(v)}$ that
is independent in the matroid $M$.
The choice of $Y=X\cap\overline{T(v)}$ yields that
\[\paren*{X \cap M}\cup\set{e}=(X\cap T(v))\cup\set{e}\cup(X\cap\overline{T(v)})\]
is independent in the matroid $M$.
We next analyze the assignment of tokens and their distribution.
Let $v_e$ be the leaf of $T$ such that $f(e) = v_e$.
Note that the number of tokens assigned to the vertices that are not on the path from $v_e$ to the root
has not changed
as neither the formula \eqref{eq:leaf} nor the formula \eqref{eq:branch}
is affected for such vertices by adding $e$ to the set $X$.
Let $v_0$ be the first vertex on the path from $v_e$ to the root such that
$\rank_{M/\overline{T(v_0)}}((X\cap T(v_0))\cup\set{e})=\rank_{M/\overline{T(v_0)}}(X\cap T(v_0))+1$.
The choice of $e$ implies that
the vertex $v_0$ exists and
is on the path from $v_e$ to $v$ (inclusive of the end vertices of the path).
The choice of $v_0$ yields that
\[\rank_{M/\overline{T(v')}}((X\cap T(v'))\cup\set{e})=\rank_{M/\overline{T(v')}}(X\cap T(v'))\]
for any vertex $v'$ on the path from $v_e$ to $v_0$ (inclusive of $v_e$ and exclusive of $v_0$).
Lemma~\ref{lm:onechange} now yields that
\[\rank_{M/\overline{T(v')}}((X\cap T(v'))\cup\set{e})=\rank_{M/\overline{T(v')}}(X\cap T(v'))+1\]
for any vertex $v'$ on the path from $v_0$ to the root (inclusive of $v_0$ and exclusive of the root).

Lemma~\ref{lm:tokens-aux} yields that
if $v'$ is a leaf or an internal branching vertex,
then the number of tokens assigned to the vertices of the subtree $T[v']$ with respect to the set $X\cup\set{e}$
is one larger than the number of tokens assigned with respect to the set $X$
if and only if $v'$ lies on the path from $v_0$ to the root (inclusive);
otherwise, the numbers of tokens assigned are the same.
Hence, all vertices except for $v_0$ are assigned the same number of tokens with respect to $X$ and with respect to $X\cup\set{e}$, and
the vertex $v_0$ is assigned one additional token.
During the procedure for the token distribution starting with the token assignment with respect to $X\cup\set{e}$,
this additional token is either kept by a vertex on the path from $v_0$ to $v$ or
by one of the bottom vertices of the edges in $P$.
Hence, the set $X\cup\set{e}$ is $\TT$-tamed as desired.
This completes the proof that the collection of $\TT$-tamed sets satisfies the augmentation axiom.
\end{proof}

In the rest of the paper,
if $M$ is a matroid and $\TT$ is a contraction$^*$-decomposition of $M$,
then $M^\TT$ denotes the matroid with the ground set $M\cup E(T)$ and
independent sets being precisely $\TT$-tamed subsets of $M\cup E(T)$.
Note that $M^\TT$ is indeed a matroid by Theorem~\ref{thm:matroid}.
Since for any set $X$ of size larger than $|E(T)|$,
the procedure for the token distribution starting with the token assignment with respect to $X$
results in assigning at least one token to the root of $T$,
every $\TT$-tamed set has size at most $\rank(M)=|E(T)|$.
Recall that the set $E(T)$ is $\TT$-tamed and so it is a basis of $M^\TT$.

We now show that the matroid $M$ is a restriction of the matroid $M^\TT$.

\begin{lemma}
\label{lm:restriction}
Let $M$ be a matroid and $\TT$ a contraction$^*$-decomposition of $M$.
Then the restriction of the matroid $M^\TT$ to the elements of $M$ is the matroid $M$ itself.
\end{lemma}

\begin{proof}
Since every $\TT$-tamed subset of the elements of $M$ is independent in $M$,
it is enough to show that every basis of $M$ is an independent set in $M^\TT$.
Let $B$ be a basis of $M$ and consider the token assignment with respect to $B$.
We show that the following holds for every leaf or internal branching vertex $v$:

\begin{center}
\emph{
The number of tokens assigned to the vertices of $T[v]$ is equal to $\rank_{M/\overline{T(v)}}(T(v))$.
}
\end{center}

\noindent Consider a leaf or an internal branching vertex $v$.
Observe that
\[\rank_{M/\overline{T(v)}} \paren*{T(v)} + \rank \paren*{\overline{T(v)}}=
\rank (M)=\rank (B)\le\rank_{M/\overline{T(v)}} (B\cap T(v))+\rank \paren*{\overline{T(v)}}\]
where ranks are calculated in $M$ unless otherwise specified.
It follows that $\rank_{M/\overline{T(v)}} (B\cap T(v))$ is at least $\rank_{M/\overline{T(v)}}(T(v))$ and
so $\rank_{M/\overline{T(v)}} (B\cap T(v))$ is equal to $\rank_{M/\overline{T(v)}}(T(v))$.
The statement above now follows from Lemma~\ref{lm:tokens-aux},
which asserts that
the number of tokens assigned to the vertices of $T[v]$
is $\rank_{M/\overline{T(v)}} (B\cap T(v))$.

We next show that every vertex of the tree $T$ except for the root has kept one token during the token distribution phase.
Let $v$ be a leaf or an internal branching vertex and
$v'$ the $\preccurlyeq$-minimal ancestor of $v$ that is a branching vertex or the root.
Lemma~\ref{lm:edges} yields that the path from $v$ to $v'$
has at most $\rank_{M/\overline{T(v)}}(T(v))-|E(T[v])|$ edges.
Since the vertices of $T[v]$ are assigned exactly $\rank_{M/\overline{T(v)}}(T(v))$ tokens,
the number of tokens of $v$ before it is processed is at least $\rank_{M/\overline{T(v)}}(T(v))-|E(T[v])|$.
It follows that every vertex on the path from $v$ to $v'$ (inclusive of $v$ but exclusive of $v'$) keeps one token.
Since this holds for every choice of the vertex $v$,
we conclude that every vertex of the tree $T$ except for the root has kept one token during the token distribution phase.
However, the number of such vertices is $\rank (M)=\lvert B\rvert$ and
the total number of tokens assigned is also $\lvert B\rvert$ by Lemma~\ref{lm:tokens}.
Hence, the root of the tree $T$ has zero tokens at the end of the token distribution phase and
so the basis $B$ is $\TT$-tamed.
\end{proof}

\subsection{Contraction-depth}
\label{subsec:decomp}

We now show that the matroid $M^\TT$ defined in the previous subsection
can indeed be viewed as the completion of $M$ preserving the contraction$^*$-depth.
We achieve this with the aid of the following three lemmata.

\begin{lemma}
\label{lm:leaf}
Let $M$ be a matroid,
$\TT=(T,f)$ a contraction$^*$-decomposition of $M$, and
$v$ a leaf of $T$.
Let $P$ be the set of edges on the path from $v$ to the root of $T$.
Then every element of $T(v)$ is a loop in the matroid $M^\TT/P$.
\end{lemma}

\begin{proof}
Recall that the set $E(T)$ is a basis of $M^\TT$.
Since the set $P$ is independent in $M^\TT$,
it is enough to show that for every $e\in T(v)$,
the set $P\cup\set{e}$ is not independent in the matroid $M^\TT$.
Suppose that $P\cup\set{e}$ is independent for some $e\in T(v)$;
note that $e$ is not a loop in $M$ (otherwise, it would be a loop in $M^\TT$, too).
Observe that with respect to the set $P\cup\set{e}$,
the tokens are assigned only to the vertices on the path $P$.
By Lemma~\ref{lm:tokens}, there are $|P|+1$ tokens assigned, however,
the number of non-root vertices on the path is only $|P|$.
Hence, the set $P\cup\set{e}$ is not $\TT$-tamed and so it is not independent in $M^\TT$.
\end{proof}

\begin{lemma}
\label{lm:root}
Let $M$ be a matroid and $\TT=(T,f)$ a contraction$^*$-decomposition of $M$.
Suppose that the root of $T$ is a branching vertex.
Let $v_1, \dots, v_k$ be all $\preccurlyeq$-maximal descendants of the root that are branching vertices or leaves, and
let $F_i$ be the set of edges contained in $\udt{T}{v_i}$, $i \in [k]$.
Then for each $i \in [k]$ the set $T(v_i)\cup F_i$ is a union of components of $M^\TT$.
\end{lemma}

\begin{proof}
Recall that $\rank (M^\TT) = \card{E(T)} = \rank (M)$.
We show that $\rank_{M^\TT}(T(v_i) \cup F_i) = \card{F_i}$ for every $i\in [k]$.
As $F_i \subseteq E(T)$ is independent, we get $\rank_{M^\TT}(T(v_i) \cup F_i) \geq \card{F_i}$.
For any subset $X \subseteq T(v_i) \cup F_i$,
all tokens are assigned to vertices in $T[v_i]$ and
so the procedure for the token distribution starting with the token assignment with respect to $X$
assigns the token to the vertices of the tree $\udt{T}{v_i}$ only.
In particular, if $X \subseteq T(v_i) \cup F_i$ is independent in $M^\TT$,
then the size of $X$ is at most $\card{F_i}$.

Observe that the sets $F_1, \ldots, F_k$ partition the edge set of the tree $T$ and
so it holds that $\card{F_1} + \dots + \card{F_k} = \card{E(T)} = \rank (M^\TT)$.
We conclude that
since the sets $T(v_i)\cup F_i$, $i \in [k]$, partition the ground set of $M^\TT$ and
the sum of their ranks is the rank of the matroid $M^\TT$,
each set $T(v_i)\cup F_i$ is a union of components of $M^\TT$.
\end{proof}

We refer to Figure~\ref{fig:lm:path} for illustration of the notation used in the next lemma.

\begin{figure}
\begin{center}
\epsfbox{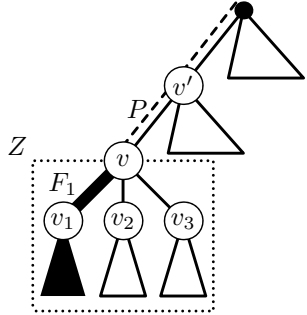}
\end{center}
\caption{Illustration of the notation used in the proof of Lemma~\ref{lm:path}.}
\label{fig:lm:path}
\end{figure}

\begin{lemma}
\label{lm:path}
Let $M$ be a matroid and $\TT=(T, f)$ a contraction$^*$-decomposition of $M$.
Let $v$ be a branching vertex or the root of $T$, and
$P$ the set of edges of $T$ contained on the path from $v$ to the root.
Moreover,
let $v_1,\ldots,v_k$ be the $\preccurlyeq$-maximal descendants of $v$ that are branching vertices or leaves, and
let $F_i$ be the union of the set of edges contained in $T[v_i]$ and the set of edges on the path from $v_i$ to $v$ for every $i\in [k]$.
Then for each $i\in [k]$ the set $T(v_i)\cup F_i$ is a union of components of $M^\TT/P$.
\end{lemma}

\begin{proof}
The proof proceeds by induction on the distance of the vertex $v$ from the root.
Suppose that we are in the setting of the lemma and
the statement holds for all possible choices of $v$ larger in $\preccurlyeq$.
Note that $T(v)=T(v_1)\cup\cdots\cup T(v_k)$ and $E(T[v])=F_1\cup\cdots\cup F_k$.
Further, let $Z$ be the union $T(v)\cup E(T[v])$.

We first argue that the following claim holds:
\emph{the set $Z$ is a union of components of the matroid $M^\TT/P$}.
Let $v'$ be the $\preccurlyeq$-minimal ancestor of $v$ that is a branching vertex;
if $v$ has no ancestor that is a branching vertex, then $v'$ is the root.
In particular, if $v$ itself is the root, then $v'$ is the root (and this is the only case when $v=v'$).
We next distinguish three cases based on
whether $v'$ is the root or not and
whether $v'$ is a branching vertex or not.
\begin{itemize}
\item If $v'$ is the root and it is not a branching vertex,
      then the set $Z$ is the set of all elements of $M^\TT$ except for those contained in $P$ (note that $P$ can be the empty set).
      Hence, the set $Z$ contains all elements of the matroid $M^\TT/P$ and the claim holds.
\item If $v'$ is the root and it is a branching vertex,
      then $Z\cup P$ is a union of components of the matroid $M^\TT$ by Lemma~\ref{lm:root}.
      Hence, the set $Z$ is a union of components of $M^\TT/P$.
\item If $v'$ is a branching vertex but not the root, we can apply induction to $v'$ as $v'\not=v$.
      Let $P'$ be the set of edges on the path from $v'$ to the root.
      The induction yields that $Z\cup (P\setminus P')$
      is a union of components of the matroid $M^\TT/P'$.
      Hence, $Z$ is a union of components of the matroid $M^\TT/P$.
\end{itemize}
The claim is now proven.

Let $F_0$ be the set of all edges of $T$ not contained in $T[v]$.
Note that $P$ is a subset of $F_0$, and if $v$ is the root, then $P=F_0=\emptyset$.
Since $Z$ is a union of components of the matroid $M^\TT/P$ and
so contracting elements not contained in $Z$ does not affect the restriction of $M^\TT/P$ to $Z$,
the restriction of the matroid $M^\TT/P$ to $Z$ and
the restriction of the matroid $M^\TT/F_0$ to $Z$ are the same.
Hence, it is enough to prove that
each $T(v_i)\cup F_i$, $i\in [k]$, is a union of components of $M^\TT/F_0$.

Consider $i\in [k]$.
Since the set $F_i$ is independent in $M^\TT/F_0$ (as the set $F_0\cup F_i$ is independent in $M^\TT$),
the rank of $T(v_i)\cup F_i$ in the matroid $M^\TT/F_0$ is at least $\lvert F_i\rvert$.
On the other hand, if $X$ is a subset of $T(v_i)\cup F_i$,
the token assignment with respect to $X\cup F_i$ assigns no tokens to a vertex of any tree $T[v_j]$, $j\in [k]\setminus\{i\}$.
Hence, if $X\cup F_0$ is independent in $M^\TT$, i.e., $\TT$-tamed,
its size is at most $\lvert F_i\rvert+\lvert F_0\rvert$.
It follows that the size of any independent subset $T(v_i)\cup F_i$ in $M^\TT/F_0$ is at most $\lvert F_i\rvert$.
We conclude that
\[\rank_{M^\TT/F_0} (T(v_i)\cup F_i)=\lvert F_i\rvert\]
for every $i\in [k]$.
Observe that any element of $\overline{T(v)}$ is a loop in the matroid $M^\TT/F_0$ by \Cref{lm:leaf}.
Since the sets $T(v_1)\cup F_1, \dots, T(v_k) \cup F_k$ partition the set of $T(v)\cup E(T[v])$,
which contains all non-loop elements of the matroid $M^\TT/F_0$, and
the sum over $i \in [k]$ of the ranks of the sets $T(v_i)\cup F_i$ is equal to the rank of $M^\TT/F_0$,
each set $T(v_i)\cup F_i$ is a union of components of $M^\TT/F_0$.
\end{proof}

We are now ready to prove that
the contraction-depth of the matroid $M^\TT$
is at most the height of the contraction$^*$-decomposition $\TT$ of a matroid $M$.

\begin{theorem}
\label{thm:depth}
Let $M$ be a matroid and $\TT = (T, f)$ a contraction$^*$-decomposition of $M$.
Then the contraction-depth of the matroid $M^\TT$ is at most the height of the tree $T$.
\end{theorem}

\begin{proof}
We prove the following claim:
\emph{Let $v$ be a leaf, a branching vertex or the root of $T$, and
      let $P$ be the set of edges on the path from $v$ to the root.
      The contraction-depth of the matroid $M^\TT/P$ restricted to $T(v)\cup E(T[v])$
      is at most the height of the tree $T[v]$.}
Once proven, the claim implies the statement of the theorem by choosing the vertex $v$ to be the root of $T$.

We prove the claim using bottom-up induction on the tree $T$ starting with the leaves.
Let $v$ be a leaf and let $P$ be the set of edges on the path from $v$ to the root.
By Lemma~\ref{lm:leaf}, every element of $T(v)$ is a loop in $M^\TT/P$.
Hence, the contraction-depth of the matroid $M^\TT/P$ restricted to $T(v)$ is one,
i.e., the height of $T[v]$.

\begin{figure}
\begin{center}
\epsfbox{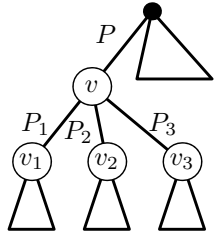}
\end{center}
\caption{Illustration of the notation used in the proof of Theorem~\ref{thm:depth}.}
\label{fig:thm:depth}
\end{figure}

Let $v$ be a branching vertex of the tree $T$ or the root of $T$.
Let $P$ be the set of edges on the path from $v$ to the root.
Moreover,
let $v_1,\dots,v_k$ be the $\preccurlyeq$-maximal descendants of $v$ that are branching vertices or leaves, and
let $P_i$ be the set of edges on the path from $v_i$ to $v$ for each $i\in [k]$.
By Lemma~\ref{lm:path}, each set $T(v_i)\cup E(T[v_i]) \cup P_i$ is a union of components in $M^\TT/P$.
Hence, the contraction-depth of $M^\TT/P$ restricted to $T(v)\cup E(T[v])$
is at most the maximum over $i \in [k]$ of the contraction-depth of $M^\TT/P$ restricted to $T(v_i)\cup E(T[v_i])\cup P_i$.
Consider $i\in [k]$.
The contraction-depth of $M^\TT/P$ restricted to $T(v_i)\cup E(T[v_i])\cup P_i$
is at most the contraction-depth of $M^\TT/(P\cup P_i)$ restricted to $T(v_i)\cup E(T[v_i])$
increased by $|P_i|$ (the elements of $P_i$ can be those that are first contracted in the definition of the contraction-depth).
By induction,
the contraction-depth of $M^\TT/(P\cup P_i)$ restricted to $T(v_i)\cup E(T[v_i])$
is at most the height of $T[v_i]$ and so
the contraction-depth of $M^\TT/P$ restricted to $T(v_i)\cup E(T[v_i])\cup P_i$
is at most the height of $T[v_i]$ increased by $\lvert P_i\rvert$,
which is at most the height of $T[v]$.
It follows that
the contraction-depth of $M^\TT/P$ restricted to $T(v)\cup E(T[v])$
is at most the height of $T[v]$ as desired.
\end{proof}

We are now ready to prove the main result of this paper.

\begin{proof}[Proof of Theorem~\ref{thm:main}.]
Suppose that the matroid $M$ consists of loops and coloops only.
If $M$ is empty, then $\csd(M) = \cd(M) = 0$.
If $M$ has at least one coloop,
then the contraction$^*$-depth $\csd(M)$ of $M$ is equal to $1$, and
if $M$ has loops only,
then the contraction$^*$-depth $\csd(M)$ of $M$ is equal to $0$.
Note that if $M$ is not empty, then
the contraction-depth of any matroid, in particular, any matroid containing $M$ as a restriction, is at least one, and
the contraction-depth of the matroid $M$ itself is equal to $1$.
Hence, the equality in the statement of the theorem holds if $M$ consists of loops only and
fails if $M$ has at least one coloop or is empty.
This establishes the theorem when $M$ has coloops and loops only.

In the rest of the proof, we assume that the matroid $M$ contains an element that is neither a loop nor a coloop.
Consider a matroid $M'$ that contains $M$ as a restriction.
Observe that $M'$ does not contain coloops and loops only.
Hence, \Cref{thm:upper} yields that the contraction-depth of such a matroid $M'$ is at least $\csd(M')+1$.
Since the contraction$^*$-depth is minor-monotone~\cite{KarKLM17},
we obtain that $\csd(M') \geq \csd(M)$.
Thus the contraction-depth of the matroid $M'$ is at least $\csd(M)+1$.
We have proven that $\csd(M)\le\cd(M')-1$ for any $M'\sqsupseteq M$.

Let $\TT$ be a contraction$^*$-decomposition of $M$ with depth $\csd(M)$.
The matroid $M^\TT$, which contains $M$ as a restriction by Lemma~\ref{lm:restriction},
has contraction-depth at most $\csd(M)+1$ by Theorem~\ref{thm:depth},
i.e.~$\cd(M^\TT)\le\csd(M)+1$.
Hence, the matroid $M^\TT$ is a matroid $M'\sqsupseteq M$ such that $\cd(M')-1\le\csd(M)$ (in fact,
the equality holds since $\csd(M)\le\cd(M')-1$ for any such matroid $M'$).
The statement of the theorem now follows.
\end{proof}

\section{Conclusion}
\label{sec:concl}

It is interesting to note that
matroids formed by loops and coloops only,
which form an exception in \Cref{thm:main},
also appear as exceptional cases in regard to other matroid width parameters.
Most prominently,
the branch-width of a graph $G$ and the branch-width of the cycle matroid of $G$
are equal unless $G$ is forest~\cite{HicM07,MazT07},
in which case the cycle matroid of $G$ consists of coloops only.
We remark that it is possible to modify the definition of contraction-depth in a way that 
the exceptional case in \Cref{thm:main} disappears.
To modify \Cref{thm:main},
consider the \emph{altered contraction-depth} of a matroid $M$, denoted by $\cd'(M)$,
which is defined recursively as follows:
\begin{itemize}
\item If $M$ has one element, then $\cd'(M)=\rank(M)$.
\item If $M$ is not connected, then $\cd'(M)$ is the maximum altered contraction-depth of a component of $M$.
\item Otherwise $\cd'(M)=1+\min\limits_{e\in M}\cd'(M/e)$, 
      i.e., $\cd'(M)$ is one plus the minimum altered contraction-depth of $M/e$
      where the minimum is taken over all elements $e$ of $M$.
\end{itemize}

Following the proof of \Cref{thm:main},
using altered contraction-depth instead of contraction-depth whenever it appears,
one can derive the next theorem.
\begin{theorem}
\label{thm:main2}
Let $M$ be a matroid.
Then the contraction$^*$-depth of $M$ is equal to the minimum altered contraction-depth
of a matroid $M'$ that contains $M$ as a restriction,
i.e.,
\[\csd(M)=\min_{M'\sqsupseteq M}\cd'(M').\]
\end{theorem}

From the broader point of view of matroid theory,
it is desirable that contraction-depth and deletion-depth are dual notions.
When one attempts to define a notion dual to altered contraction-depth,
which we dub here altered deletion-depth and denote by $\dd'(M)$,
it would necessarily be defined as follows:
\begin{itemize}
\item If $M$ consists of a single loop, then $\dd'(M)=1$.
\item If $M$ consists of a single coloop, then $\dd'(M)=0$.
\item If $M$ is not connected, then $\dd'(M)$ is the maximum altered deletion-depth of a component of $M$.
\item Otherwise $\dd'(M)=1+\min\limits_{e\in M}\dd'(M/e)$, 
      i.e., $\dd'(M)$ is one plus the minimum altered deletion-depth of $M\setminus e$
      where the minimum is taken over all elements $e$ of $M$.
\end{itemize}
However, it seems very unnatural that
the altered deletion-depth of a matroid consisting of a single loop should be larger than
that of a matroid consisting of a single coloop.

\section*{Acknowledgement}

The authors would like to thank both anonymous reviewers for their careful reading of this manuscript and many detailed comments,
which helped to improve the presentation of the results.

\bibliographystyle{bibstyle}
\bibliography{cdepth}

\end{document}